\documentclass[11pt]{article}

\usepackage[a4paper,margin=2.5cm]{geometry}
\usepackage[T1]{fontenc}
\usepackage{amsmath,amssymb,amsthm}
\usepackage{microtype}
\usepackage{graphicx}
\usepackage{tikz}
\usetikzlibrary{arrows.meta,calc,decorations.pathreplacing}
\usepackage[hidelinks]{hyperref}

\newcommand{\R}{\mathbb R}
\newcommand{\abs}[1]{\lvert #1\rvert}
\newcommand{\norm}[1]{\lVert #1\rVert}
\newcommand{\dist}{\operatorname{dist}}
\newcommand{\diam}{\operatorname{diam}}
\newcommand{\sgn}{\operatorname{sgn}}

\newtheorem{theorem}{Theorem}[section]
\newtheorem{lemma}[theorem]{Lemma}
\newtheorem{proposition}[theorem]{Proposition}
\newtheorem{corollary}[theorem]{Corollary}
\theoremstyle{definition}
\newtheorem{definition}[theorem]{Definition}
\theoremstyle{remark}
\newtheorem{remark}[theorem]{Remark}

\title{ A quantitative convergence rate for torsional creep
to the distance function as p tends to infinity}

\author{Farid Bozorgnia\thanks{Department of Mathematics, New Uzbekistan University, Tashkent, Uzbekistan. Email: \texttt{f.bozorgnia@newuu.uz}}
\and
Masoud Bayrami-Aminlouee\thanks{Department of Mathematical Sciences, Sharif University of Technology, P.O.\ Box 11365-9415, Tehran, Iran; \texttt{masoud.bayrami1990@sharif.edu}. School of Mathematics, Institute for Research in Fundamental Sciences (IPM), P.O.\ Box 19395-5746, Tehran, Iran; \texttt{aminlouee@ipm.ir}}
\and
Khudoyor Mamayusupov\thanks{Department of Mathematics, New Uzbekistan University, Tashkent, Uzbekistan. Email: \texttt{k.mamayusupov@newuu.uz}}
}
\date{}

\begin{document}
\maketitle

\begin{abstract}
We prove that the Dirichlet $p$-torsion function $u_p$ converges to the
distance to the boundary $d$ at the rate $O(1/p)$ in the uniform norm, on
every bounded domain and with a constant depending only on the dimension and
the diameter. The upper bound comes from a radial barrier with its pole at a
boundary point, admissible for $p>n$ without any boundary regularity; the
lower bound comes from an inscribed ball. On the ball the error equals
$(1+\log n)/p+O(p^{-2})$, so the order $1/p$ cannot be improved in general.
Under a uniform exterior ball condition an annular barrier gives
$u_p\le K_\Omega^{1/(p-1)}d$ for every $p>1$, with $K_\Omega$ explicit; on
convex sets the constant is the diameter, and a multiple of $d$ is a
supersolution whenever $-\Delta d$ has a positive distributional lower bound.
On $C^2$ domains the same barrier and the classical gradient maximum principle
give $\norm{\nabla u_p}_{L^\infty}^{p-1}\le K_\Omega$. When $-\Delta d$ has
finite total variation, which we prove for Lipschitz domains with a uniform
exterior ball, the gradients converge in $L^q$ at the rate $O(p^{-1/2})$ for
$q\le2$ and $O(p^{-1/q})$ for $q\ge2$; these exponents are not claimed to be
sharp, and the explicit ball profile shows that the gradients do not converge
uniformly.
\end{abstract}


\section{Introduction}\label{sec:intro}
 Let $n\ge1$, let $\Omega\subset\R^n$ be an open bounded, connected domain,  and let $p>1$. We denote by $u_p\in W^{1,p}_0(\Omega)$
the unique weak solution of the Dirichlet $p$-torsion problem
\begin{equation}\label{eq:torsion}
\begin{cases}
-\Delta_p u_p = 1 & \text{in } \Omega,\\
u_p = 0 & \text{on } \partial\Omega,
\end{cases}
\end{equation}
where $\Delta_p u := \operatorname{div}(\abs{\nabla u}^{p-2}\nabla u)$, and we set
\[
d(x):=\dist(x,\partial\Omega),\qquad
\rho_\Omega:=\norm{d}_{L^\infty(\Omega)},\qquad
D:=\diam(\overline\Omega),\qquad
L_p:=\norm{\nabla u_p}_{L^\infty(\Omega)},
\]
allowing $L_p=+\infty$; finiteness is asserted only where a gradient bound is
proved. Throughout, $\gamma:=p/(p-1)$ denotes the conjugate exponent,
\begin{equation}\label{eq:Apn}
A_{p,n}:=\frac{p-1}{p}\,n^{-1/(p-1)}\in(0,1]
\end{equation}
is the central value of the $p$-torsion function of the unit ball
(Theorem~\ref{thm:ball}), and, for $\rho>0$,
\begin{equation}\label{eq:Cnrho}
C(n,\rho):=2\max\bigl\{e^{-1},\,\rho^{2}\log\max\{1,\rho\}\bigr\}+(2+2\log n)\,\rho
\end{equation}
is the constant of the lower estimate (Proposition~\ref{thm:lower}),
non-decreasing in $\rho$.

Bhattacharya--DiBenedetto--Manfredi~\cite{BDM} and Kawohl~\cite{Kaw}
established the uniform convergence $u_p\to d$ on $\overline\Omega$ as
$p\to\infty$, and Ishii--Loreti~\cite{IL} developed the associated variational
and viscosity framework. For positive right-hand sides,
Buccheri--Leonori--Rossi~\cite{BLR} proved strong convergence of $\nabla u_p$
to the gradient of the limiting function in every $L^q(\Omega)$, $q<\infty$,
and exhibited the failure of $L^\infty$ convergence of the gradients. The
qualitative limit and the strong gradient convergence are thus known; what is
not known is their rate. This paper proves: (i) a global $O(1/p)$ uniform rate
on every bounded domain, with a constant depending only on $n$ and the
diameter (Theorem~\ref{thm:universal-rate}); (ii) upper barriers linear in
$d$, with explicit constants, under an exterior ball condition, convexity, or
a positive distributional lower bound on $-\Delta d$
(Theorem~\ref{thm:main}, Corollary~\ref{cor:convex},
Theorem~\ref{thm:dbarrier}); (iii) an explicit bound on
$\norm{\nabla u_p}_{L^\infty}^{p-1}$ on $C^2$ domains
(Theorem~\ref{thm:gradient-bound}); (iv) $L^q$ rates for the gradients
whenever $-\Delta d$ has finite total variation, in particular on convex and
on $C^2$ domains (Theorem~\ref{thm:gradient-rate}); and (v) on the ball, the
exact uniform error and a second-order expansion uniform up to the medial axis
(Theorem~\ref{thm:ball}). The explicit ball profile~\cite{VdBB2014}, its first
two expansion coefficients~\cite{FB2026}, and the gradient maximum
principle~\cite{PP,Sperb,Kaw1987} are known results that we use, not claim.

Quantitative rates as $p\to\infty$ were previously available only for the
homogeneous problem: Bungert~\cite{Bun} proved that $p$-harmonic functions
converge to the $\infty$-harmonic function with the same boundary values at
the rate $p^{-1/4}$, and $p^{-1/2}$ under a positive-gradient hypothesis. The
torsion problem permits the order $1/p$ with explicit constants because its
limit is explicit: $d$ is $1$-Lipschitz and satisfies $\abs{\nabla d}=1$
almost everywhere, whereas the $\infty$-harmonic limit is characterized only
implicitly, through comparison with cones. Radial solutions of $-\Delta_pw=1$
are explicit as well (Lemma~\ref{lem:radial}), so $u_p$ can be compared from
both sides with $A_{p,n}\,d^{\gamma}$ and $K^{1/(p-1)}d$
(Lemma~\ref{lem:ball-comparison}, Proposition~\ref{thm:upper}), and the rate
follows from $K^{1/(p-1)}=1+O(1/p)$ and $d^{\gamma}-d=O(1/p)$
(Lemma~\ref{lem:elementary}). The only point of contact with the $p$-harmonic
mechanism is the exponent $(p-n)/(p-1)$ of the barrier of
Theorem~\ref{thm:universal-rate}, which is that of the H\"older cones
in~\cite{Bun}; here the resulting defect $d^{(p-n)/(p-1)}-d$ is again $O(1/p)$,
so no rate is lost. Apart from the formal expansion of
Fayolle--Belyaev~\cite{FB2026}, which comes without error bounds, we are not
aware of a previous rate for the torsion problem in any normalization; the
order $1/p$ cannot be improved, and on the unit ball the leading constant is
$1+\log n$ (Corollary~\ref{cor:ball-constant}).

\subsection*{Main results}
Our first result is a global uniform rate that assumes nothing beyond
boundedness of $\Omega$. Its upper half comes from a radial solution with its
singularity at a boundary point, which is an admissible Sobolev comparison
function for $p>n$; its lower half comes from an inscribed ball.
\begin{theorem}[Global uniform rate on bounded domains]\label{thm:universal-rate}
Let $\Omega\subset\R^n$ be a bounded domain and set
\[
M_{n,D}:=2(n-1)+\log\max\{1,D^n/n\},
\qquad
p_0:=\max\{2n,\,1+M_{n,D}\}.
\]
Then, for every $p>n$,
\begin{equation}\label{eq:universal-upper}
u_p(x)\le
\frac{p-1}{p-n}\,n^{-1/(p-1)}D^{n/(p-1)}
d(x)^{(p-n)/(p-1)}
\qquad(x\in\Omega),
\end{equation}
and, for every $p\ge p_0$,
\begin{equation}\label{eq:universal-rate}
\norm{u_p-d}_{L^\infty(\Omega)}
\le\frac{4\rho_\Omega M_{n,D}+4(n-1)+C(n,\rho_\Omega)}{p}
\le\frac{C_{n,D}}{p},
\end{equation}
where $C_{n,D}:=2DM_{n,D}+4(n-1)+C(n,D/2)$ and the second inequality uses $\rho_\Omega\le D/2$ and the monotonicity of
$C(n,\cdot)$ in \eqref{eq:Cnrho}. The order $1/p$ cannot be improved in
general (Corollary~\ref{cor:ball-constant}).
\end{theorem}

Under a uniform exterior ball condition (Definition~\ref{def:ext-ball}) the
upper estimate becomes linear in the distance and holds for every $p>1$, with
an explicit constant; this is less general than
Theorem~\ref{thm:universal-rate} but stronger where it applies.
\begin{theorem}[Rate under the exterior ball condition]\label{thm:main}
Let $\Omega\subset\R^n$ be a bounded domain satisfying
Definition~\ref{def:ext-ball} with radius $r_e>0$, and set
\begin{equation}\label{eq:Komega}
K_\Omega:=\frac{(D+r_e)^n-r_e^{\,n}}{n\,r_e^{\,n-1}},
\qquad
K_\Omega^+:=\max\{1,K_\Omega\}.
\end{equation}
Then, for every $p>1$,
\begin{equation}\label{eq:main-pointwise}
A_{p,n}\,d(x)^{\gamma}
\ \le\ u_p(x)\ \le\ K_\Omega^{1/(p-1)}\,d(x)
\qquad(x\in\Omega),
\end{equation}
and consequently
\begin{equation}\label{eq:main-rate}
\norm{u_p-d}_{L^\infty(\Omega)}
\ \le\ \frac{4\rho_\Omega\log K_\Omega^+ + C(n,\rho_\Omega)}{p}
\qquad\text{for every }p\ge\max\{2,\,1+\log K_\Omega^+\},
\end{equation}
with $C(n,\rho_\Omega)$ as in \eqref{eq:Cnrho}.
\end{theorem}

The upper estimates of this paper are ordered by their hypotheses.
Theorem~\ref{thm:universal-rate} assumes nothing beyond boundedness.
Theorem~\ref{thm:main} assumes a uniform exterior ball, which every convex
domain and every $C^{1,1}$ domain has, and gives an estimate linear in $d$
for every $p>1$; on convex sets the constant reduces to the diameter
(Corollary~\ref{cor:convex}). Finally, a constant multiple of $d$ is itself
a supersolution precisely when $-\Delta d$ has a positive distributional
lower bound, $-\Delta d\ge\lambda>0$ (Theorem~\ref{thm:dbarrier}). This
condition should be distinguished from mean convexity: by a classical
characterization (Gromov, Giga--Pisante, Lewis--Li--Li, Psaradakis; see
Remark~\ref{rem:LLL}), $-\Delta d\ge0$ is equivalent to weak mean convexity of
a $C^2$ boundary, but the barrier needs the strict inequality, which uniform
mean convexity implies (Proposition~\ref{prop:umc}) and weak mean convexity,
strict convexity and convexity alone do not (Remark~\ref{rem:LLL}). Since
Theorem~\ref{thm:universal-rate} needs no convexity at all, these hypotheses
affect only the constants.

The exterior barrier also controls the boundary normal derivative, which is
the new ingredient of the next theorem; the classical gradient maximum
principle, verified in Lemma~\ref{lem:grad-max} for the degenerate
$p$-Laplacian, extends that control to the whole domain.
\begin{theorem}[Gradient bound at the $(p-1)$-power scale]\label{thm:gradient-bound}
Let $\Omega\subset\R^n$ be a bounded $C^2$ domain and let $r_e>0$ be any
admissible radius in Definition~\ref{def:ext-ball}; such a radius exists for
every bounded $C^2$ domain. Then, for every $p>1$,
\begin{equation}\label{eq:gradient-bound}
\norm{\nabla u_p}_{L^\infty(\Omega)}^{\,p-1}\ \le\ K_\Omega.
\end{equation}
\end{theorem}
By \eqref{eq:K-expansion} below, $K_\Omega$ decreases as $r_e$ increases, so
the bound is sharpest for the largest admissible radius.

\subsection*{Related literature}
Radial $p$-torsion on annuli is treated by Bueno--Ercole~\cite{BE2015}, and
the explicit Dirichlet profile on a ball is recorded by van den
Berg--Bucur~\cite{VdBB2014}. Variable-exponent and Orlicz-type
torsional-creep limits were studied by P\'erez-Llanos--Rossi~\cite{PLR} and
Bocea--Mih\u{a}ilescu~\cite{BocMih}, and Buccheri--Leonori~\cite{BL2021}
studied the large-$p$ limit of large solutions, which blow up at the boundary
and are not covered by the comparison arguments used here. Quantitative
estimates for the $p$-torsional rigidity of convex sets appear in
\cite{AMPS,EMSD,FGL}. The $p$-Poisson equation is also used as a distance
approximation in geometry processing~\cite{BF2015,FB2018}; the formal
expansion $u_p=d+a/p+b/p^2+\cdots$ of Fayolle--Belyaev~\cite{FB2026}, with
explicit interval and ball coefficients in their Appendices~A and~B, is
recovered in Section~\ref{sec:ball} with a rigorous uniform remainder.

\section{Preliminaries}\label{sec:prelim}
Throughout, $\Omega\subset\R^n$ is a bounded domain, repeated indices are
summed, and $C$ denotes a positive constant whose dependence is indicated
explicitly when it matters and which otherwise depends only on $\Omega$ and
$n$; its value may change from one occurrence to the next, unless it carries a
subscript or arguments, in which case it is fixed. A bounded domain is a
\emph{Lipschitz domain} if, locally and after a rigid motion, $\Omega$ is the
subgraph $\{x_n<g(x')\}$ of a Lipschitz function $g$; it is a $C^1$ or $C^2$
domain if $g$ can be taken of class $C^1$ or $C^2$. Bounded convex domains are
Lipschitz domains.

\subsection{Weak solutions and comparison}
For every $p>1$, problem
\eqref{eq:torsion} has a unique weak solution $u_p\in W^{1,p}_0(\Omega)$,
obtained by minimizing
$v\mapsto\frac1p\int_\Omega\abs{\nabla v}^p-\int_\Omega v$ over
$W^{1,p}_0(\Omega)$. The minimizer satisfies
\begin{equation}\label{eq:weak}
\int_\Omega\abs{\nabla u_p}^{p-2}\nabla u_p\cdot\nabla\varphi\,dx
=\int_\Omega\varphi\,dx
\qquad\text{for every }\varphi\in W^{1,p}_0(\Omega).
\end{equation}
The zero boundary condition is understood in this Sobolev sense. For $p>n$,
the zero extension of $u_p$ belongs to $W^{1,p}(\R^n)$ and has a continuous
representative vanishing on $\partial\Omega$, by Morrey's inequality.

\begin{lemma}[Positivity]\label{lem:positivity}
$u_p\ge0$ almost everywhere in $\Omega$.
\end{lemma}
\begin{proof}
Put $u_p^-:=\max\{-u_p,0\}\in W^{1,p}_0(\Omega)$ and test \eqref{eq:weak} with
$\varphi=u_p^-$. Since $\nabla u_p^-=-\nabla u_p$ on $\{u_p<0\}$ and
$\nabla u_p^-=0$ almost everywhere elsewhere,
\[
\int_\Omega\abs{\nabla u_p^-}^p\,dx
=-\int_\Omega\abs{\nabla u_p}^{p-2}\nabla u_p\cdot\nabla u_p^-\,dx
=-\int_\Omega u_p^-\,dx.
\]
The left-hand side is non-negative and the right-hand side is non-positive;
hence both vanish, and $u_p^-=0$ by the Poincar\'e inequality.
\end{proof}

Interior regularity for degenerate equations with bounded right-hand side
(DiBenedetto~\cite{DiB}, Tolksdorf~\cite{Tol}) gives
$u_p\in C^{1,\alpha}_{\mathrm{loc}}(\Omega)$ for every $p>1$; we use this
continuous representative throughout. If $\partial\Omega$ is $C^2$, then
Lieberman~\cite{Lib} gives $u_p\in C^{1,\alpha}(\overline\Omega)$ for some
$\alpha=\alpha(p,\Omega)$, so $\nabla u_p$ extends continuously to
$\overline\Omega$ and $L_p$ is attained. No uniformity in $p$ of $\alpha$ or of
the $C^{1,\alpha}$ norm is used anywhere. On the open set where
$\nabla u_p\ne0$ the equation is locally uniformly elliptic with smooth
dependence on the gradient, so $u_p$ is $C^\infty$ there by standard
bootstrapping; this is used only in Lemma~\ref{lem:grad-max}.

Every upper estimate in this paper is obtained by exhibiting a non-negative
supersolution, and the lower estimate by exhibiting a subsolution on an
inscribed ball.
\begin{lemma}[Weak comparison]\label{lem:comparison}
Let $\Omega\subset\R^n$ be a bounded domain and $p>1$.
\begin{enumerate}
\item[(i)] Let $w\in W^{1,p}(\Omega)$ satisfy $w\ge0$ almost everywhere and
\begin{equation}\label{eq:supersol}
\int_\Omega\abs{\nabla w}^{p-2}\nabla w\cdot\nabla\varphi\,dx
\ \ge\ \int_\Omega\varphi\,dx
\qquad\text{for every non-negative }\varphi\in C_c^\infty(\Omega).
\end{equation}
Then $u_p\le w$ almost everywhere in $\Omega$. No regularity of
$\partial\Omega$ is assumed.
\item[(ii)] Let $B\subset\Omega$ be an open ball and let $w\in W^{1,p}_0(B)$
satisfy
\[
\int_B\abs{\nabla w}^{p-2}\nabla w\cdot\nabla\varphi\,dx
\ \le\ \int_B\varphi\,dx
\qquad\text{for every non-negative }\varphi\in W^{1,p}_0(B).
\]
Then $w\le u_p$ almost everywhere in $B$.
\end{enumerate}
\end{lemma}
\begin{proof}
Both parts follow from the weak comparison principle for the $p$-Laplacian
\cite[Chapter~3]{Lin}: if $u,w\in W^{1,p}(U)$ satisfy
$-\Delta_pu\le-\Delta_pw$ weakly against non-negative test functions in
$W^{1,p}_0(U)$ and $(u-w)^+\in W^{1,p}_0(U)$, then $u\le w$ almost everywhere
in $U$. What has to be verified in each case is the boundary hypothesis
$(u-w)^+\in W^{1,p}_0(U)$, and this is where the present setting departs
from the standard one: $\partial\Omega$ is not assumed regular and the
barriers used below do not belong to $W^{1,p}_0(\Omega)$.

(i) Since $\abs{\nabla w}^{p-2}\nabla w\in L^{p'}$, both sides of
\eqref{eq:supersol} are continuous on $W^{1,p}_0(\Omega)$, and non-negative
$C_c^\infty$ functions are dense in its positive cone (approximate, take
positive parts, mollify), so \eqref{eq:supersol} holds for every non-negative
$\varphi\in W^{1,p}_0(\Omega)$. For the boundary hypothesis, choose
$u_j\in C_c^\infty(\Omega)$ converging to $u_p$ in $W^{1,p}(\Omega)$. Because
$w\ge0$, each $(u_j-w)^+$ vanishes outside the support of $u_j$; it is a
compactly supported $W^{1,p}$ function, hence lies in $W^{1,p}_0(\Omega)$.
Moreover $(u_j-w)^+\to(u_p-w)^+$ in $L^p(\Omega)$, because truncation is
$1$-Lipschitz, and the gradients
$\nabla(u_j-w)^+=\chi_{\{u_j>w\}}\nabla(u_j-w)$ are bounded in
$L^p(\Omega)$; hence a subsequence converges weakly in $W^{1,p}(\Omega)$,
necessarily to $(u_p-w)^+$, and $(u_p-w)^+\in W^{1,p}_0(\Omega)$ because this
closed subspace is weakly closed. The comparison principle with $U=\Omega$
gives $u_p\le w$.

(ii) The boundary hypothesis $(w-u_p)^+\in W^{1,p}_0(B)$ is verified as in
(i), with the roles of the two functions exchanged: choose
$w_j\in C_c^\infty(B)$ converging to $w$ in $W^{1,p}(B)$; since $u_p\ge0$
almost everywhere (Lemma~\ref{lem:positivity}), each $(w_j-u_p)^+$ vanishes
outside the support of $w_j$ and hence lies in $W^{1,p}_0(B)$, and the same
weak-compactness argument gives $(w-u_p)^+\in W^{1,p}_0(B)$. Since the zero
extension of any function in $W^{1,p}_0(B)$ is admissible in \eqref{eq:weak},
$u_p$ satisfies $-\Delta_pu_p=1$ weakly in $B$, and the comparison principle
with $U=B$, applied to $w$ and $u_p$, gives $w\le u_p$ in $B$.
\end{proof}
\subsection{Radial solutions}\label{subsec:radial}
For $R\ge0$ set
\begin{equation}\label{eq:F}
F_R(s):=\frac{R^n-s^n}{n\,s^{n-1}}\qquad(s>0).
\end{equation}
Then $F_R>0$ on $(0,R)$, $F_R<0$ on $(R,\infty)$, and
\begin{equation}\label{eq:F-decreasing}
F_R'(s)=\frac1n\Bigl((1-n)R^ns^{-n}-1\Bigr)<0\qquad(s>0),
\end{equation}
so $F_R$ is strictly decreasing. In the lemma below $F_R$ is the radial flux
of a solution of $-\Delta_pw=1$, and $R$ is the radius at which this flux
vanishes.
\begin{lemma}[Radial solutions]\label{lem:radial}
Let $p>1$, $c\in\R^n$, $R\ge0$, let $I\subset(0,\infty)\setminus\{R\}$ be an
open interval, and let $\phi\in C^1(I)$ satisfy
\begin{equation}\label{eq:radial-ode}
\phi'(s)=\sgn\bigl(F_R(s)\bigr)\,\abs{F_R(s)}^{1/(p-1)}\qquad(s\in I);
\end{equation}
such $\phi$ exist and are unique up to an additive constant, the right-hand
side being continuous on $I$. Then $w(z):=\phi(\abs{z-c})$ is $C^\infty$ on the
open set $S:=\{z\in\R^n:\ \abs{z-c}\in I\}$ and satisfies $-\Delta_pw=1$
classically in $S$, with $\abs{\nabla w}^{p-2}\nabla w=F_R(\abs{z-c})\,(z-c)/\abs{z-c}$.
\end{lemma}
\begin{proof}
On $I$ the function $F_R$ is smooth and does not vanish, so $\phi'$, and hence
$\phi$, is smooth on $I$; and $z\mapsto\abs{z-c}$ is smooth on $S$ because
$0\notin I$. Thus $w\in C^\infty(S)$. By \eqref{eq:radial-ode}:
\[
\abs{\phi'}^{p-2}\phi'=\sgn(F_R)\abs{F_R}=F_R \quad \textit{on} \, \, I.
\]
 For a radial function
$w=\phi(s)$, $s=\abs{z-c}$, with unit radial field $\hat n:=(z-c)/s$, one has
$\nabla w=\phi'(s)\hat n$, hence
$\abs{\nabla w}^{p-2}\nabla w=\abs{\phi'(s)}^{p-2}\phi'(s)\,\hat n=F_R(s)\,\hat n$,
and $\operatorname{div}\bigl(G(s)\hat n\bigr)=G'(s)+\tfrac{n-1}{s}G(s)
=s^{1-n}\bigl(s^{n-1}G(s)\bigr)'$ for every $G\in C^1(I)$. Therefore, in $S$,
\[
\Delta_pw=s^{1-n}\bigl(s^{n-1}F_R(s)\bigr)'
=s^{1-n}\Bigl(\frac{R^n-s^n}{n}\Bigr)'=-1. \qedhere
\]
\end{proof}

\subsection{The distance function}
Two properties of $d$ are used repeatedly: it has zero Sobolev trace, so
that $u_p-d$ is an admissible test function in \eqref{eq:weak} and $\theta d$
an admissible comparison function in Lemma~\ref{lem:comparison}; and it
satisfies $\abs{\nabla d}=1$ almost everywhere, so that the $p$-Laplacian
flux of $\theta d$ is simply $\theta^{p-1}\nabla d$.
\begin{lemma}\label{lem:d-in-W1p}
Let $\Omega\subset\R^n$ be a bounded open set and $1<p<\infty$. Then
$d\in W^{1,\infty}(\Omega)\cap W^{1,p}_0(\Omega)$, and $d$ satisfies the
eikonal equation $\abs{\nabla d}=1$ almost everywhere in $\Omega$.
\end{lemma}
\begin{proof}
That $d$ is $1$-Lipschitz with $\abs{\nabla d}=1$ almost everywhere is
classical \cite[Chapter~3]{CS}. For $0<\varepsilon<\rho_\Omega$ the Lipschitz
function $d_\varepsilon:=(d-\varepsilon)^+$ has compact support
$\{d\ge\varepsilon\}\subset\Omega$, so its mollifications at small scales lie
in $C_c^\infty(\Omega)$ and $d_\varepsilon\in W^{1,p}_0(\Omega)$; by the chain
rule for truncations \cite[Lemma~7.6]{GT},
$\norm{\nabla d_\varepsilon-\nabla d}^p_{L^p}\le\abs{\{0<d\le\varepsilon\}}\to0$,
and $d_\varepsilon\to d$ uniformly, so $d\in W^{1,p}_0(\Omega)$.
\end{proof}

A uniform gradient bound yields pointwise control by the distance to the
boundary; this is used in Corollary~\ref{cor:gradient-limit}.
\begin{lemma}[Segment bound]\label{lem:segment}
Let $\Omega\subset\R^n$ be a bounded domain and let
$w\in C^1(\Omega)\cap C(\overline\Omega)$ satisfy $w=0$ on $\partial\Omega$ and
$L:=\norm{\nabla w}_{L^\infty(\Omega)}<\infty$. Then
$\abs{w(x)}\le L\,d(x)$ for every $x\in\Omega$.
\end{lemma}
\begin{proof}
Fix $x\in\Omega$ and $y\in\partial\Omega$ with $\abs{x-y}=d(x)$, and set
$x_t:=x+t(y-x)$. For $0\le t<1$, $x_t\in B(x,d(x))\subset\Omega$, where $w$ is
$C^1$, so $w(x)-w(x_s)=-\int_0^s\nabla w(x_t)\cdot(y-x)\,dt$ for $s<1$; letting
$s\uparrow1$ and using $w(y)=0$ gives $\abs{w(x)}\le L\abs{x-y}=L\,d(x)$.
\end{proof}

\subsection{The exterior ball condition}
\begin{definition}[Uniform exterior ball condition]\label{def:ext-ball}
A bounded domain $\Omega\subset\R^n$ satisfies the \emph{uniform exterior ball
condition with radius $r_e>0$} if, for every $y\in\partial\Omega$, there
exists $c_y\in\R^n$ such that
\[
\abs{c_y-y}=r_e,
\qquad B(c_y,r_e)\cap\Omega=\emptyset,
\]
where $B(c_y,r_e)$ denotes the open ball; $c_y$ is not assumed to be unique
and no boundary regularity is assumed in this definition. If
$\partial\Omega$ is of class $C^1$ with outer unit normal $\nu$, then
necessarily $c_y=y+r_e\nu(y)$, and the condition takes the familiar form
\begin{equation}\label{eq:ext-ball}
B\bigl(y+r_e\nu(y),r_e\bigr)\cap\Omega=\emptyset
\qquad\text{for every }y\in\partial\Omega .
\end{equation}
\end{definition}
Every bounded convex domain satisfies Definition~\ref{def:ext-ball} for every
$r_e>0$, with no regularity of the boundary: if $\nu$ is the outer unit normal
of a supporting hyperplane at $y$, then $B(y+r_e\nu,r_e)$ is contained in the
open half-space $\{z:(z-y)\cdot\nu>0\}$, which is disjoint from $\Omega$. Every
bounded $C^2$ domain admits a uniform exterior-ball radius
\cite[Section~14.6]{GT}; more generally, a bounded domain admits uniform
interior and exterior balls if and only if its boundary is of class $C^{1,1}$
\cite[Theorem~2.6]{Dal2018}. If $r_e$ is admissible, so is every $0<r<r_e$,
since $B(y+r(c_y-y)/r_e,\,r)\subset B(c_y,r_e)$. Whenever an exterior-ball
hypothesis is imposed, $r_e$ denotes a fixed admissible radius. The constant
$K_\Omega$ of \eqref{eq:Komega} depends on this choice: by the binomial theorem,
\begin{equation}\label{eq:K-expansion}
K_\Omega=\frac1n\sum_{m=1}^{n}\binom nm D^{\,m}\,r_e^{\,1-m},
\end{equation}
which decreases as $r_e$ increases (strictly for $n\ge2$; $K_\Omega=D$ for
$n=1$) with limit $D$ as $r_e\to\infty$, so $r_e$ should be taken as large
as possible.
 
Two elementary consequences of Definition~\ref{def:ext-ball} are used
below. First, $\abs{z-c_y}>r_e$ for every $z\in\Omega$ and every
$y\in\partial\Omega$, since equality would place part of a neighborhood of
$z$ inside the exterior ball. Second, if $y$ is a nearest boundary point of
$x\in\Omega$, then
\begin{equation}\label{eq:tangency}
\abs{x-c_y}=r_e+d(x).
\end{equation}
Indeed, $\abs{x-c_y}\le\abs{x-y}+\abs{y-c_y}=d(x)+r_e$, while the open balls
$B(x,d(x))\subset\Omega$ and $B(c_y,r_e)\subset\R^n\setminus\Omega$ are
disjoint, whence $\abs{x-c_y}\ge d(x)+r_e$.
 
\begin{lemma}[Finite total variation of the distance Laplacian]\label{lem:tv-exterior-ball}
Let $\Omega\subset\R^n$ be a bounded Lipschitz domain satisfying
Definition~\ref{def:ext-ball} with radius $r_e>0$. Then $\mu:=-\Delta d$ is a
signed Radon measure on $\Omega$ with
\[
\abs{\mu}(\Omega)\ \le\ 2c_\Omega+\frac{2n}{r_e}\abs{\Omega}\ <\ \infty,
\]
where $c_\Omega$ is the boundary-layer constant of \eqref{eq:layer} below.
\end{lemma}
\begin{proof}
\emph{Step 1: $d$ is a minimum of smooth functions.}
For $y\in\partial\Omega$ let $c_y$ be as in Definition~\ref{def:ext-ball} and
put $g_y(x):=\abs{x-c_y}-r_e$. Fix $x\in\Omega$. Since $\abs{x-c_y}>r_e$,
$g_y(x)$ is the distance from $x$ to the ball $B(c_y,r_e)$, and since that
ball is contained in $\R^n\setminus\Omega$,
\[
g_y(x)=\dist\bigl(x,B(c_y,r_e)\bigr)\ \ge\ \dist(x,\R^n\setminus\Omega)=d(x)
\qquad\text{for every }y\in\partial\Omega,
\]
with equality when $y$ is a nearest boundary point of $x$, by the tangency
identity \eqref{eq:tangency}. Hence
\begin{equation}\label{eq:d-as-min}
d(x)=\min_{y\in\partial\Omega}g_y(x)\qquad(x\in\Omega).
\end{equation}
 
\emph{Step 2: semiconcavity and the sign of $\mu+\frac n{r_e}\mathcal L^n$.}
Each $g_y$ is smooth on $\Omega$ with
\[
D^2g_y(x)=\frac{1}{\abs{x-c_y}}\bigl(\mathrm{Id}-\hat n\,\hat n^{\top}\bigr)
\ \le\ \frac1{r_e}\,\mathrm{Id},
\qquad \hat n:=\frac{x-c_y}{\abs{x-c_y}},
\]
because $\abs{x-c_y}>r_e$. Thus on every ball $B\Subset\Omega$ the functions
$g_y-\tfrac1{2r_e}\abs{x}^2$ are concave, and by \eqref{eq:d-as-min} so is
their infimum $f:=d-\tfrac1{2r_e}\abs{x}^2$. A function that is concave on
every ball $B\Subset\Omega$ satisfies $\partial_{\xi\xi}f\le0$ in
$\mathcal D'(\Omega)$ for every unit vector $\xi$, because
$\int f\,\partial_{\xi\xi}\varphi\,dx$ is the limit of
$\int h^{-2}\bigl(f(x+h\xi)-2f(x)+f(x-h\xi)\bigr)\varphi\,dx\le0$ for
non-negative $\varphi\in C_c^\infty(B)$ and small $h>0$, and a partition of
unity extends this to $\Omega$. As $\partial_{\xi\xi}\bigl(\tfrac1{2r_e}\abs{x}^2\bigr)=r_e^{-1}$
for $\abs\xi=1$, this gives
\[
\langle\partial_{\xi\xi}d,\varphi\rangle\ \le\ \frac1{r_e}\int_\Omega\varphi\,dx
\qquad\text{for every unit }\xi\text{ and every non-negative }
\varphi\in C_c^\infty(\Omega),
\]
which is the meaning of $D^2d\le r_e^{-1}\mathrm{Id}$ in $\mathcal D'(\Omega)$.
Taking the trace, that is, summing over an orthonormal basis $e_1,\dots,e_n$,
yields $\langle\Delta d,\varphi\rangle\le\frac n{r_e}\int_\Omega\varphi\,dx$,
or equivalently
\[
\Bigl\langle\mu+\frac n{r_e}\,\mathcal L^n,\ \varphi\Bigr\rangle\ \ge\ 0
\qquad\text{for every non-negative }\varphi\in C_c^\infty(\Omega).
\]
Hence $\widetilde\mu:=\mu+\frac n{r_e}\,\mathcal L^n$ is a non-negative
distribution, therefore a non-negative Radon measure on $\Omega$, and
$\mu=\widetilde\mu-\frac n{r_e}\,\mathcal L^n$ is a locally finite signed Radon
measure. It remains to show that $\widetilde\mu(\Omega)<\infty$.
 
\emph{Step 3: volume of the boundary layer.}
Since $\partial\Omega$ is compact and Lipschitz, it is covered by finitely many
charts in which it is the graph $x_n=g(x')$ of a Lipschitz function, and
there is $s_0>0$ such that every $x\in\Omega$ with $d(x)<s_0$ lies in the
cylinder of a chart containing a nearest boundary point $(y',g(y'))$ of $x$.
For such $x$,
\[
\abs{x_n-g(x')}\le\abs{x_n-g(y')}+\abs{g(y')-g(x')}
\le\bigl(1+\operatorname{Lip}g\bigr)\,d(x),
\]
so $\{0<d<s\}$ lies, in each chart, in a vertical strip of thickness
$2(1+\operatorname{Lip}g)s$ over a bounded base. Summing over the charts,
\begin{equation}\label{eq:layer}
\abs{\{0<d<s\}}\ \le\ c_\Omega\,s\qquad(0<s<s_0),
\end{equation}
with $c_\Omega$ depending only on the covering.
 
\emph{Step 4: total mass of $\widetilde\mu$.}
Fix $0<\sigma<s_0/2$ and, for $0<\eta<\sigma$, set
\[
\varphi_{\eta}:=\min\Bigl\{\frac{(d-\eta)^+}{\sigma},\,1\Bigr\},
\]
a Lipschitz function with values in $[0,1]$, supported in the compact set
$\{d\ge\eta\}\subset\Omega$, with $\nabla\varphi_\eta=\sigma^{-1}\nabla d$
on $\{\eta<d<\eta+\sigma\}$ and $\nabla\varphi_\eta=0$ almost everywhere
elsewhere. Since $d\in W^{1,\infty}(\Omega)$, the definition of $\mu=-\Delta d$
gives
\begin{equation}\label{eq:test-tilde-mu}
\int_\Omega\varphi\,d\widetilde\mu
=\int_\Omega\nabla d\cdot\nabla\varphi\,dx+\frac n{r_e}\int_\Omega\varphi\,dx
\qquad\text{for }\varphi\in C_c^\infty(\Omega),
\end{equation}
and \eqref{eq:test-tilde-mu} extends to $\varphi=\varphi_\eta$ by
mollification, since the mollifications converge uniformly and in
$W^{1,1}(\Omega)$ with supports in a fixed compact set, on which
$\widetilde\mu$ is finite, and $\nabla d\in L^\infty(\Omega)$. Using $\abs{\nabla d}=1$ almost everywhere,
$0\le\varphi_\eta\le1$, and \eqref{eq:layer} with $\eta+\sigma<2\sigma<s_0$,
\[
\int_\Omega\varphi_{\eta}\,d\widetilde\mu
=\frac{\abs{\{\eta<d<\eta+\sigma\}}}{\sigma}+\frac n{r_e}\int_\Omega\varphi_\eta\,dx
\ \le\ \frac{c_\Omega(\eta+\sigma)}{\sigma}+\frac n{r_e}\abs{\Omega}
\ \le\ 2c_\Omega+\frac n{r_e}\abs{\Omega}.
\]
As $\eta\downarrow0$ the functions $\varphi_\eta$ increase pointwise to
$\min\{d/\sigma,1\}$, which equals $1$ on $\{d\ge\sigma\}$; by monotone
convergence,
\[
\widetilde\mu\bigl(\{d\ge\sigma\}\bigr)
\ \le\ \int_\Omega\min\{d/\sigma,1\}\,d\widetilde\mu
\ \le\ 2c_\Omega+\frac n{r_e}\abs{\Omega},
\]
independently of $\sigma$. Letting $\sigma\downarrow0$ gives
$\widetilde\mu(\Omega)\le2c_\Omega+n\abs{\Omega}/r_e<\infty$, and therefore
\[
\abs{\mu}(\Omega)\ \le\ \widetilde\mu(\Omega)+\frac n{r_e}\abs{\Omega}
\ \le\ 2c_\Omega+\frac{2n}{r_e}\abs{\Omega}\ <\ \infty . \qedhere
\]
\end{proof}

\subsection{Elementary large-\texorpdfstring{$p$}{p} estimates}
The barriers produce factors such as $K^{1/(p-1)}$ and $d^{\gamma}$. The
following estimates convert them to the additive $1/p$ scale.
\begin{lemma}\label{lem:elementary}
Let $\rho>0$ and $\delta:=1/(p-1)$.
\begin{enumerate}
\item[(i)] For all $p\ge2$,
$\sup_{0\le r\le\rho}\abs{r^{\gamma}-r}\le C_2(\rho)/p$, where
$C_2(\rho):=2\max\{e^{-1},\,\rho^2\log(\max\{1,\rho\})\}$.
\item[(ii)] If $K\ge1$ then $K^{1/(p-1)}\le1+\dfrac{4\log K}{p}$ for every
$p\ge\max\{2,\,1+\log K\}$.
\item[(iii)] $A_{p,n}=\dfrac{p-1}{p}\,n^{-1/(p-1)}\ \ge\ 1-\dfrac{2+2\log n}{p}$ for
all $p\ge2$ and $n\ge1$.
\end{enumerate}
\end{lemma}
\begin{proof}
(i) If $0<r\le1$, then $1-e^{-t}\le t$ with $t=\delta\abs{\log r}$ gives
$0\le r-r^{1+\delta}\le\delta/e$. If $1\le r\le\rho$, then $e^t-1\le te^t$
with $t=\delta\log r$ and $\delta\le1$ gives
$0\le r^{1+\delta}-r\le\delta\rho^2\log\rho$. Conclude with $\delta\le2/p$.
(ii) Since $e^t\le1+2t$ for $0\le t\le1$, the condition $p-1\ge\log K$ gives
$K^{1/(p-1)}\le1+2\log K/(p-1)$, and $2/(p-1)\le4/p$ for $p\ge2$.
(iii) For $p\ge2$, $-\log(1-1/p)\le\frac1{p-1}\le\frac2p$ and
$\frac{\log n}{p-1}\le\frac{2\log n}{p}$, so the negative logarithm of
$A_{p,n}$ is at most $(2+2\log n)/p$; conclude with $e^{-t}\ge1-t$.
\end{proof}

\section{The lower estimate}\label{sec:lower}
The pointwise lower estimate follows by comparing $u_p$ with the radial
$p$-torsion function on the inscribed ball $B(x,d(x))$. It is the special
case for the $p$-Laplacian of Magnanini--Poggesi~\cite[Lemma~4.1]{MP}, which is
stated there for a class of quasilinear operators and with the source
normalized to $n$ instead of $1$; we include the short proof because the
explicit constant is used throughout.
\begin{lemma}[Inscribed-ball comparison]\label{lem:ball-comparison}
Let $\Omega\subset\R^n$ be a bounded domain and $p>1$. Then
\begin{equation}\label{eq:pointwise-lower}
u_p(x)\ \ge\ A_{p,n}\,d(x)^{\gamma}
\qquad\text{for every }x\in\Omega.
\end{equation}
\end{lemma}
\begin{proof}
Fix $x\in\Omega$, put $r:=d(x)$ and $B:=B(x,r)\subset\Omega$. The function
\[
w(y):=A_{p,n}\bigl(r^{\gamma}-\abs{y-x}^{\gamma}\bigr)
\]
lies in $C^1(\overline B)\cap W^{1,p}_0(B)$ and equals $\phi(\abs{y-x})$ with
$\phi(s):=A_{p,n}(r^\gamma-s^\gamma)$. Since $A_{p,n}\gamma=n^{-1/(p-1)}$ and
$\gamma-1=1/(p-1)$,
\[
\phi'(s)=-n^{-1/(p-1)}s^{1/(p-1)}=\sgn\bigl(F_0(s)\bigr)\abs{F_0(s)}^{1/(p-1)},
\qquad F_0(s)=-\frac sn ,
\]
so Lemma~\ref{lem:radial} with $R=0$ gives $-\Delta_pw=1$ in $B\setminus\{x\}$
with flux $\abs{\nabla w}^{p-2}\nabla w=-(y-x)/n$, a smooth field of divergence
$-1$; integrating by parts,
$\int_B\abs{\nabla w}^{p-2}\nabla w\cdot\nabla\varphi\,dy=\int_B\varphi\,dy$
for every $\varphi\in W^{1,p}_0(B)$. Lemma~\ref{lem:comparison}(ii) gives
$u_p\ge w$ in $B$ (almost everywhere, hence everywhere by continuity), and
evaluation at the center gives \eqref{eq:pointwise-lower}.
\end{proof}

Lemma~\ref{lem:elementary} converts the pointwise comparison into the
additive estimate used in Theorems~\ref{thm:universal-rate} and~\ref{thm:main}.
\begin{proposition}[Additive lower estimate]\label{thm:lower}
Let $\Omega\subset\R^n$ be a bounded domain. Then, for every $x\in\Omega$ and
every $p\ge2$,
\[
u_p(x)\ \ge\ d(x)-\frac{C(n,\rho_\Omega)}{p},
\]
with $C(n,\rho_\Omega)$ as in \eqref{eq:Cnrho}.
\end{proposition}
\begin{proof}
Put $c_n:=2+2\log n$, so that $A_{p,n}\ge1-c_n/p$ by
Lemma~\ref{lem:elementary}(iii), and recall $0<A_{p,n}\le1$. By
Lemma~\ref{lem:elementary}(i),
$\norm{d^{\gamma}-d}_{L^\infty}\le C_2/p$ with $C_2=C_2(\rho_\Omega)$. Hence
\[
u_p-d\ \ge\ A_{p,n}d^{\gamma}-d
=A_{p,n}\bigl(d^{\gamma}-d\bigr)-(1-A_{p,n})d
\ \ge\ -\frac{C_2}{p}-\frac{c_n\rho_\Omega}{p}. \qedhere
\]
\end{proof}

\section{Upper estimates by comparison}\label{sec:upper}
Each upper bound is obtained from an explicit non-negative supersolution
of $-\Delta_pw=1$ through Lemma~\ref{lem:comparison}(i). The barriers are ordered from
the most general to the sharpest: the radial profile of Lemma~\ref{lem:radial}
with its pole at a boundary point needs only boundedness, and $p>n$ for the
pole to be an admissible singularity (Theorem~\ref{thm:universal-rate}); the
same profile with its pole at the center of an exterior ball gives an estimate
linear in $d$ for every $p>1$ (Theorem~\ref{thm:main}), with the diameter as
constant in the convex case; and a multiple of $d$ itself, available when
$-\Delta d$ has a positive distributional lower bound, gives the sharpest
constant where it applies.

\subsection{A boundary-centered radial barrier: proof of Theorem~\ref{thm:universal-rate}}
\label{subsec:universal-upper}
\begin{proof}[Proof of Theorem~\ref{thm:universal-rate}]
Fix $p>n$, $y\in\partial\Omega$, and $R>D$, and put
\[
V_{y,R}(z):=\int_0^{\abs{z-y}}F_R(s)^{1/(p-1)}\,ds
\qquad(z\in\overline\Omega),
\]
where $\abs{z-y}\le D<R$ and $F_R>0$. Since
$0\le F_R(s)^{1/(p-1)}\le (R^n/n)^{1/(p-1)}s^{-(n-1)/(p-1)}$ and
$(n-1)/(p-1)<1$, $V_{y,R}$ is continuous and non-negative on $\overline\Omega$
with $V_{y,R}(y)=0$; by Lemma~\ref{lem:radial} with $c=y$ and $I=(0,R)$ it is
smooth with $-\Delta_pV_{y,R}=1$ in $\Omega\subset\{0<\abs{z-y}<R\}$; and
$\nabla V_{y,R}\in L^p(\Omega)$ because
\[
\int_\Omega\abs{\nabla V_{y,R}}^p\,dz
\le \abs{\partial B(0,1)}(R^n/n)^{p/(p-1)}
\int_0^R s^{-(n-1)/(p-1)}\,ds<\infty .
\]
Hence $V_{y,R}\in W^{1,p}(\Omega)$ satisfies \eqref{eq:supersol} with equality,
and Lemma~\ref{lem:comparison}(i) gives $u_p\le V_{y,R}$ in $\Omega$ (almost
everywhere, hence everywhere by continuity).

For a given $x\in\Omega$, choose $y\in\partial\Omega$ with
$\abs{x-y}=d(x)$. Then, with $\beta_p:=(p-n)/(p-1)$,
\[
u_p(x)\le V_{y,R}(x)
\le (R^n/n)^{1/(p-1)}\int_0^{d(x)}s^{-(n-1)/(p-1)}\,ds
=\frac{(R^n/n)^{1/(p-1)}}{\beta_p}\,d(x)^{\beta_p}.
\]
Letting $R\downarrow D$ proves \eqref{eq:universal-upper}.

For the additive error assume $p\ge p_0=\max\{2n,1+M_{n,D}\}$ and set
\[
\delta:=\frac1{p-1},\qquad \varepsilon:=(n-1)\delta,
\qquad B_p:=\frac{(D^n/n)^\delta}{1-\varepsilon},
\]
so that \eqref{eq:universal-upper} reads $u_p\le B_p\,d^{1-\varepsilon}$. We
split the error as $u_p-d\le(B_p-1)d+B_p(d^{1-\varepsilon}-d)$ and bound the
prefactor defect $B_p-1$ and the exponent defect $d^{1-\varepsilon}-d$
separately. Since $0\le\varepsilon\le1/2$ we have $-\log(1-\varepsilon)\le2\varepsilon$, and
\[
\log B_p
=-\log(1-\varepsilon)+\delta\log(D^n/n)
\le\delta\bigl(2(n-1)+\log\max\{1,D^n/n\}\bigr)=\delta M_{n,D}\le1 .
\]
Consequently $B_p\le e$ and $B_p-1\le2\delta M_{n,D}\le4M_{n,D}/p$. For
$0<r\le1$, since $1-e^{-t}\le t$ and $\sup_{0<r\le1}r^{a}\abs{\log r}=1/(ea)$,
\[
0\le r^{1-\varepsilon}-r
=r^{1-\varepsilon}\bigl(1-r^{\varepsilon}\bigr)
\le\varepsilon r^{1-\varepsilon}\abs{\log r}
\le\frac{\varepsilon}{e(1-\varepsilon)}
\le\frac{4(n-1)}{ep},
\]
whereas $r^{1-\varepsilon}-r\le0$ for $r\ge1$. Therefore
\[
u_p-d\ \le\ \frac{4\rho_\Omega M_{n,D}}{p}+e\cdot\frac{4(n-1)}{ep}
=\frac{4\rho_\Omega M_{n,D}+4(n-1)}{p}.
\]
Combining this with Proposition~\ref{thm:lower}, which applies because
$p\ge2n\ge2$, gives \eqref{eq:universal-rate}.
\end{proof}

\subsection{The exterior annulus: proof of Theorem~\ref{thm:main}}
\label{subsec:pointwise-upper}
Let $\Omega$ satisfy Definition~\ref{def:ext-ball} with radius $r_e>0$ and
let $R>D+r_e$; for radial $p$-torsion profiles on annuli see also
Bueno--Ercole~\cite{BE2015}.
\begin{lemma}[Exterior barrier]\label{lem:barrier}
For $y\in\partial\Omega$ let $c_y$ be as in Definition~\ref{def:ext-ball}, let
$\mathcal A_y:=\{z:\ r_e<\abs{z-c_y}<R\}$, and put
\begin{equation}\label{eq:W}
W_{y,R}(z):=\int_{r_e}^{\abs{z-c_y}}F_R(s)^{1/(p-1)}\,ds,
\qquad z\in\overline{\mathcal A_y}.
\end{equation}
Then $\Omega\subset\mathcal A_y$, the function $W_{y,R}$ is non-negative,
belongs to $C^1(\overline{\mathcal A_y})\cap C^\infty(\mathcal A_y)$, solves
$-\Delta_pW_{y,R}=1$ in $\mathcal A_y$, and $u_p\le W_{y,R}$ everywhere in
$\Omega$.
\end{lemma}

\begin{figure}[htbp]
\centering
\resizebox{0.88\textwidth}{!}{%
\begin{tikzpicture}[scale=1,>=Latex,every node/.style={font=\small}]
  \coordinate (c) at (0,0);        
  \coordinate (y) at (-1.1,0);     
  \coordinate (x) at (-2.0,0);     
  \fill[gray!11,even odd rule] (c) circle (3.8) (c) circle (1.1);
  \draw[dashed,gray!60] (c) circle (3.8);
  \node[gray!50!black,anchor=south] at (0,3.85) {$\partial B(c_y,R)$};
  \node[gray!50!black] at (-2.3,2.45) {$\mathcal A_y$};
  \draw[->,gray!60!black] (c) -- (50:3.8) node[pos=0.55,right=1pt] {$R$};
  \fill[blue!13] (-2.1,0) ellipse (1.0 and 1.25);
  \draw[very thick,blue!55!black] (-2.1,0) ellipse (1.0 and 1.25);
  \node[blue!55!black] at (-2.2,-1.05) {$\Omega$};
  \draw[blue!55!black,thin] (-3.35,1.45) -- (-2.81,0.88);
  \node[blue!55!black,anchor=south east] at (-3.3,1.4) {$\partial\Omega$};
  \draw[densely dashed,blue!55!black] (x) circle (0.9);
  \draw[blue!55!black,thin] (-3.05,-1.35) -- (-2.78,-0.45);
  \node[blue!55!black,anchor=north east] at (-3.0,-1.35) {$B(x,d(x))$};
  \fill[red!12] (c) circle (1.1);
  \draw[very thick,red!60!black] (c) circle (1.1);
  \node[red!60!black,anchor=north] at (0.15,-0.55) {$B(c_y,r_e)$};
  \draw[gray!70,thin] (x) -- (c);
  \draw[->,very thick] (y) -- ++(0.62,0);
  \node[anchor=north] at (-0.78,-0.1) {$\nu(y)$};
  \fill (c) circle (1.7pt) node[anchor=west,xshift=2pt] {$c_y$};
  \fill (y) circle (1.7pt) node[anchor=south,xshift=-6pt,yshift=3pt] {$y$};
  \fill (x) circle (1.7pt) node[above=3pt] {$x$};
  \draw[gray!60,densely dotted] (x) -- (-2.0,-2.3);
  \draw[gray!60,densely dotted] (y) -- (-1.1,-2.3);
  \draw[gray!60,densely dotted] (c) -- (0,-2.3);
  \draw[decorate,decoration={brace,mirror,amplitude=4pt}] (-2.0,-2.3) -- (-1.1,-2.3)
     node[midway,below=5pt] {$d(x)$};
  \draw[decorate,decoration={brace,mirror,amplitude=4pt}] (-1.1,-2.3) -- (0,-2.3)
     node[midway,below=5pt] {$r_e$};
  \begin{scope}[shift={(5.0,-3.0)}]
    \draw[->] (0,0) -- (7.2,0) node[right] {$s=\abs{z-c_y}$};
    \draw[->] (0,0) -- (0,4.9) node[above] {$W_{y,R}(s)$};
    \draw[thick,blue!55!black] plot[smooth] coordinates
      {(1.200,0.000) (1.470,0.340) (1.740,0.664) (2.010,0.972) (2.280,1.267)
       (2.550,1.549) (2.820,1.820) (3.090,2.079) (3.360,2.328) (3.630,2.567)
       (3.900,2.795) (4.170,3.013) (4.440,3.221) (4.710,3.418) (4.980,3.604)
       (5.250,3.779) (5.520,3.941) (5.790,4.088) (6.060,4.219) (6.330,4.329)
       (6.600,4.400)};
    \draw[red!60!black,thick,densely dashed] (1.2,0) -- (3.9,3.498);
    \node[anchor=east,align=right,font=\scriptsize] at (7.1,0.75)
      {dashed line: tangent at $s=r_e$,\\ slope $F_R(r_e)^{1/(p-1)}$};
    \draw[gray!70] (6.0,4.40) -- (7.0,4.40);
    \node[gray!50!black,anchor=south,font=\scriptsize] at (6.55,4.44) {$W_{y,R}'(R)=0$};
    \draw[gray!60,densely dotted] (3.0,0) -- (3.0,2.332);
    \fill[blue!55!black] (3.0,1.994) circle (1.6pt);
    \fill[red!60!black]  (3.0,2.332) circle (1.6pt);
    \draw[blue!55!black,thin] (3.0,1.994) -- (3.45,1.45);
    \node[blue!55!black,anchor=west,font=\scriptsize] at (3.45,1.42) {$W_{y,R}(x)\ \ge\ u_p(x)$};
    \draw[red!60!black,thin] (3.0,2.332) -- (2.75,2.72);
    \node[red!60!black,anchor=east,font=\scriptsize] at (2.85,2.78) {$F_R(r_e)^{1/(p-1)}\,d(x)$};
    \foreach \xx/\lab in {1.2/$r_e$, 3.0/$r_e+d(x)$, 6.6/$R$}
      {\draw (\xx,0.08) -- (\xx,-0.08) node[below] {\lab};}
  \end{scope}
\end{tikzpicture}}
\caption{Left: the exterior ball $B(c_y,r_e)$ (red) touches $\partial\Omega$ at $y$; for
$x$ with nearest boundary point $y$, the inscribed ball $B(x,d(x))$ (dashed) touches it at
the same point, so $\abs{x-c_y}=r_e+d(x)$ \eqref{eq:tangency}; the barrier lives on the
shaded annulus $\mathcal A_y$. Right: the profile $W_{y,R}$ of \eqref{eq:W} against
$s=\abs{z-c_y}$ ($n=2$, $p=4$, $r_e=1$, $R=4$; axes rescaled). It is concave because $F_R$
decreases, so it lies below its tangent at $r_e$, which gives
$u_p(x)\le W_{y,R}(x)\le F_R(r_e)^{1/(p-1)}d(x)$ (Proposition~\ref{thm:upper}).}
\label{fig:exterior-annulus}
\end{figure}

\begin{proof}
Figure~\ref{fig:exterior-annulus} summarizes the geometry. For $z\in\Omega$,
$\abs{z-c_y}>r_e$ as noted after Definition~\ref{def:ext-ball}, and
$\abs{z-c_y}\le\abs{z-y}+r_e\le D+r_e<R$; so $\Omega\subset\mathcal A_y$. Since
$F_R>0$ on $(0,R)$, Lemma~\ref{lem:radial} with $c=c_y$, $I=(0,R)$ and
$\phi(s)=\int_{r_e}^sF_R^{1/(p-1)}$ gives $W_{y,R}\in C^\infty(\mathcal A_y)$
and $-\Delta_pW_{y,R}=1$ in $\mathcal A_y$; $W_{y,R}\in C^1(\overline{\mathcal A_y})$
because $F_R^{1/(p-1)}$ is continuous on $[r_e,R]$, and $W_{y,R}\ge0$ because
the integrand is. Its restriction to $\Omega$ is thus a non-negative function
in $W^{1,\infty}(\Omega)$ satisfying \eqref{eq:supersol} with equality, and
Lemma~\ref{lem:comparison}(i) gives $u_p\le W_{y,R}$ in $\Omega$ (almost
everywhere, hence everywhere by continuity).
\end{proof}

Anchoring the barrier at a nearest boundary point identifies its radial
increment with $d(x)$.
\begin{proposition}[Exterior-ball upper estimate]\label{thm:upper}
Let $\Omega\subset\R^n$ be a bounded domain satisfying
Definition~\ref{def:ext-ball} with radius $r_e>0$, and let $K_\Omega$,
$K_\Omega^+$ be as in \eqref{eq:Komega}. Then, for every $p>1$,
\[
u_p(x)\ \le\ K_\Omega^{1/(p-1)}\,d(x)
\qquad\text{for every }x\in\Omega,
\]
and consequently
$u_p\le d+4\rho_\Omega\log K_\Omega^+/p$ in $\Omega$ for every
$p\ge\max\{2,\,1+\log K_\Omega^+\}$.
\end{proposition}
\begin{proof}
Fix $x\in\Omega$ and choose $y\in\partial\Omega$ with $\abs{x-y}=d(x)$. By
\eqref{eq:tangency}, $\abs{x-c_y}=r_e+d(x)$. Let $R>D+r_e$. By
Lemma~\ref{lem:barrier} and \eqref{eq:F-decreasing},
\[
u_p(x)\ \le\ W_{y,R}(x)
=\int_{r_e}^{r_e+d(x)}F_R(s)^{1/(p-1)}\,ds
\ \le\ F_R(r_e)^{1/(p-1)}\,d(x).
\]
Letting $R\downarrow D+r_e$ gives $F_R(r_e)\to K_\Omega$ and hence
$u_p(x)\le K_\Omega^{1/(p-1)}d(x)$. The second assertion follows from
$K_\Omega^{1/(p-1)}\le (K_\Omega^+)^{1/(p-1)}\le1+4\log K_\Omega^+/p$
(Lemma~\ref{lem:elementary}(ii)) and $d\le\rho_\Omega$.
\end{proof}

\begin{proof}[Proof of Theorem~\ref{thm:main}]
The two-sided bound \eqref{eq:main-pointwise} is
Lemma~\ref{lem:ball-comparison} together with Proposition~\ref{thm:upper},
and \eqref{eq:main-rate} follows from the second assertion of
Proposition~\ref{thm:upper} and Proposition~\ref{thm:lower}.
\end{proof}

\begin{remark}[Scaling]\label{rem:scaling}
For $a>0$ and $\Omega_a:=a\Omega$ one has
$u_{p,\Omega_a}(ax)=a^{\gamma}u_{p,\Omega}(x)$,
$d_{\Omega_a}(ax)=a\,d_\Omega(x)$ and $K_{\Omega_a}=a\,K_\Omega$, so
\eqref{eq:main-pointwise} is invariant under dilation, since
$a^{1/(p-1)}\cdot a=a^{\gamma}$, while the additive constant in
\eqref{eq:main-rate} changes by a logarithm of the scale, as is intrinsic to
the $p\to\infty$ expansion. For fixed $D$ and $n\ge2$,
$K_\Omega\sim D^nr_e^{\,1-n}/n$ as $r_e\downarrow0$ by \eqref{eq:K-expansion},
so this constant grows like $\log(1/r_e)$, whereas the constant of
Theorem~\ref{thm:universal-rate} does not depend on $r_e$.
\end{remark}

\subsection{Convex domains}\label{subsec:convex}
Every bounded convex domain admits every $r_e>0$ as an exterior radius, and
$K_\Omega\downarrow D$ as $r_e\uparrow\infty$ by \eqref{eq:K-expansion}.
Letting $r_e\to\infty$ in Proposition~\ref{thm:upper} therefore gives the
following estimate, in which only the diameter enters; it is the estimate
one would obtain directly by comparing with the one-dimensional profile of a
supporting slab.
\begin{corollary}[Convex domains]\label{cor:convex}
Let $\Omega\subset\R^n$ be a bounded open convex set and $D^+:=\max\{1,D\}$.
Then, for every $p>1$,
\begin{equation}\label{eq:slab}
u_p(x)\ \le\ D^{1/(p-1)}\,d(x)
\qquad\text{for every }x\in\Omega,
\end{equation}
and
$\norm{u_p-d}_{L^\infty(\Omega)}\le\bigl(4\rho_\Omega\log D^+ + C(n,\rho_\Omega)\bigr)/p$
for every $p\ge\max\{2,1+\log D^+\}$.
\end{corollary}
\begin{proof}
For each $r_e>0$, Proposition~\ref{thm:upper} gives
$u_p\le K_\Omega^{1/(p-1)}d$; let $r_e\to\infty$ and use
\eqref{eq:K-expansion}. The additive bound follows from
Lemma~\ref{lem:elementary}(ii) with $K=D^+$ and Proposition~\ref{thm:lower}.
\end{proof}

\subsection{A distance-function supersolution}\label{subsec:dbarrier}
When the distance function is strictly superharmonic in the distributional
sense, a scalar multiple of $d$ is itself a supersolution. On the domains
where this applies it gives a sharper constant than the exterior annulus.
\begin{theorem}[Distance supersolution estimate]\label{thm:dbarrier}
Let $\Omega\subset\R^n$ be a bounded domain and suppose there is $\lambda>0$
with
\begin{equation}\label{eq:kappa}
-\Delta d\ \ge\ \lambda
\qquad\text{in }\mathcal D'(\Omega),
\end{equation}
that is,
$\int_\Omega\nabla d\cdot\nabla\varphi\,dx\ge\lambda\int_\Omega\varphi\,dx$ for
every non-negative $\varphi\in C_c^\infty(\Omega)$. Then, for every $p>1$,
\begin{equation}\label{eq:dbarrier-pointwise}
u_p(x)\ \le\ \lambda^{-1/(p-1)}\,d(x)
\qquad\text{for every }x\in\Omega.
\end{equation}
In particular $u_p\le d$ for every $p>1$ if $\lambda\ge1$, while if
$\lambda<1$ then $u_p\le d+4\rho_\Omega\log(1/\lambda)/p$ in $\Omega$ for
every $p\ge\max\{2,\,1+\log(1/\lambda)\}$.
\end{theorem}
\begin{proof}
Set $\theta:=\lambda^{-1/(p-1)}$ and $w:=\theta d$. By
Lemma~\ref{lem:d-in-W1p}, $w\in W^{1,\infty}(\Omega)$, $w\ge0$, and
$\abs{\nabla d}=1$ almost everywhere, so for every non-negative
$\varphi\in C_c^\infty(\Omega)$,
\[
\int_\Omega\abs{\nabla w}^{p-2}\nabla w\cdot\nabla\varphi\,dx
=\theta^{p-1}\int_\Omega\nabla d\cdot\nabla\varphi\,dx
\ \ge\ \theta^{p-1}\lambda\int_\Omega\varphi\,dx
=\int_\Omega\varphi\,dx .
\]
Lemma~\ref{lem:comparison}(i) gives \eqref{eq:dbarrier-pointwise} almost
everywhere, hence everywhere, since $u_p$ and $d$ are continuous. If
$\lambda\ge1$ then $\theta\le1$. If $\lambda<1$,
Lemma~\ref{lem:elementary}(ii) with $K:=1/\lambda>1$ gives
$\theta\le1+4\log(1/\lambda)/p$ for $p\ge\max\{2,1+\log(1/\lambda)\}$, and
$d\le\rho_\Omega$.
\end{proof}

\begin{remark}\label{rem:dbarrier-structure}
Estimate \eqref{eq:dbarrier-pointwise} has the form of
\eqref{eq:main-pointwise} and \eqref{eq:slab} with $K_\Omega$, respectively
$D$, replaced by $1/\lambda$, and is dilation invariant since $\lambda$ scales
like an inverse length. On the unit ball with $n\ge2$,
Proposition~\ref{prop:umc} gives $\lambda=n-1\ge1$ and hence $u_p\le d$, which
is the inequality of Theorem~\ref{thm:ball}, whereas
Corollary~\ref{cor:convex} only gives $u_p\le 2^{1/(p-1)}d$.
\end{remark}

The hypothesis \eqref{eq:kappa} is a curvature condition. Principal
curvatures $\kappa_1,\dots,\kappa_{n-1}$ of a $C^2$ boundary are taken with
respect to the outer unit normal and signed so that spheres have positive
curvature, and $H:=\sum_i\kappa_i$ is the non-averaged mean curvature. In these
conventions, Lewis--Li--Li~\cite[Theorem~1.6 and Proposition~3.9]{LewisLiLi2012}
give the following; see also Psaradakis~\cite[Theorem~3.4]{Psaradakis2013}.
\begin{proposition}[Uniform mean-convexity criterion of Lewis--Li--Li]\label{prop:umc}
Let $n\ge2$ and let $\Omega\subset\R^n$ be a bounded $C^2$ domain with
$H\ge H_0>0$ on $\partial\Omega$. Then $-\Delta d\ge H_0$ in
$\mathcal D'(\Omega)$. In particular this holds if $\Omega$ is uniformly
convex, $\kappa_i\ge\kappa_0>0$ for every $i$, with $H_0=(n-1)\kappa_0$.
\end{proposition}

The mechanism is contained in the distance-function theory of
\cite{CrastaMalusa,MM,LiNir} and \cite[Ch.~3]{CS}: for a bounded $C^2$
domain, $d$ is locally semiconcave, so the singular part of $-\Delta d$,
carried by the cut locus (a null set), is non-negative, while off the cut locus
$d$ is $C^2$ and
\begin{equation}\label{eq:reg-set-formula}
-\Delta d(x)=\sum_{i=1}^{n-1}\frac{\kappa_i(y(x))}{1-d(x)\,\kappa_i(y(x))}
\ \ge\ H(y(x)),
\end{equation}
$y(x)$ being the nearest boundary point; the inequality holds because every
denominator is positive and
$\frac{\kappa_i}{1-t\kappa_i}-\kappa_i=\frac{t\kappa_i^2}{1-t\kappa_i}\ge0$.

\begin{remark}\label{rem:LLL}
The characterization of the weakly mean convex $C^2$ domains as those for
which $-\Delta d\ge0$ in $\mathcal D'(\Omega)$ goes back to
Gromov~\cite[\S\,1/2]{Gromov1991}; it was established independently, and by
different methods, by Giga--Pisante~\cite{GigaPisante2013},
Lewis--Li--Li~\cite[Theorem~1.7]{LewisLiLi2012} and
Psaradakis~\cite[Corollary~3.6]{Psaradakis2013}, the last in the context of
$L^1$ Hardy inequalities and together with the quantitative lower bound of
Proposition~\ref{prop:umc}~\cite[Theorem~3.4]{Psaradakis2013}. Strict
positivity in \eqref{eq:kappa}
is exactly what the barrier requires: $\theta d$ is a weak supersolution of
$-\Delta_pw=1$ if and only if $-\Delta d\ge\theta^{1-p}$ in
$\mathcal D'(\Omega)$, so if \eqref{eq:kappa} holds for no $\lambda>0$ then no
positive multiple of $d$ serves as a barrier. Uniform mean convexity cannot be
weakened to strict convexity: the boundary of $\{\sum_ix_i^4<1\}$ is strictly
convex but all its principal curvatures vanish where it meets a coordinate
axis, so $\inf_{\partial\Omega}H=0$. Convexity alone is not enough either: on
a rectangular box $-\Delta d$ is purely singular, concentrated on the ridge
set, and no $\lambda>0$ is admissible.
\end{remark}

\section{The gradient bound}\label{sec:gradient-bound}
The gradient maximum principle of Payne--Philippin~\cite{PP},
Sperb~\cite{Sperb} and Kawohl~\cite{Kaw1987} (see also Mariconda--Treu~\cite{MT}
for minima of convex functionals) reduces the global gradient bound to a
boundary estimate. The lemma below is therefore not new; we include a proof
adapted to the degenerate $p$-Laplacian because what matters here is that it
holds without any sign condition on the boundary curvature and with no
constant depending on $p$.
\begin{lemma}[Gradient maximum principle]\label{lem:grad-max}
Let $\Omega\subset\R^n$ be a bounded $C^2$ domain and $p>1$. Then
\begin{equation}\label{eq:grad-max}
\max_{\overline\Omega}\abs{\nabla u_p}=\max_{\partial\Omega}\abs{\nabla u_p}.
\end{equation}
\end{lemma}
\begin{proof}
The maximum is attained because $u_p\in C^{1}(\overline\Omega)$, and $L_p>0$
because $u_p\not\equiv0$. Suppose the maximum is attained at an interior
point $x_0$. Then $\nabla u_p(x_0)\ne0$ and, since $\nabla u_p$ is continuous
on $\overline\Omega$, there is a connected open set $U\ni x_0$ with
$\overline U\subset\Omega\cap\{\nabla u_p\ne0\}$; on $\overline U$ we have
$0<m\le\abs{\nabla u_p}\le L_p$ for some $m$, the equation is uniformly
elliptic with smooth dependence on $\nabla u_p$, and $u_p$ is $C^\infty$ in a
neighborhood of $\overline U$. Write $a(\xi):=\abs{\xi}^{p-2}\xi$ and
\[
a_{ij}:=\frac{\partial a_i}{\partial\xi_j}(\nabla u_p)
=\abs{\nabla u_p}^{p-2}
\Bigl(\delta_{ij}+(p-2)\frac{(u_p)_i(u_p)_j}{\abs{\nabla u_p}^{2}}\Bigr),
\]
a symmetric matrix with eigenvalues $\abs{\nabla u_p}^{p-2}$ and
$(p-1)\abs{\nabla u_p}^{p-2}$; since $m\le\abs{\nabla u_p}\le L_p$ on
$\overline U$, $(a_{ij})$ is bounded and uniformly elliptic on $U$ for every
$p>1$. Differentiating
$\partial_ia_i(\nabla u_p)=-1$ with respect to $x_k$ gives
\begin{equation}\label{eq:diff-eq}
\partial_i\bigl(a_{ij}(u_p)_{jk}\bigr)=0
\qquad\text{in }U,\ k=1,\dots,n.
\end{equation}
Put $v:=\tfrac12\abs{\nabla u_p}^2$, so $v_j=\sum_k(u_p)_k(u_p)_{kj}$. Using
\eqref{eq:diff-eq},
\begin{equation}\label{eq:subsolution}
\partial_i\bigl(a_{ij}v_j\bigr)
=\sum_k\Bigl[(u_p)_k\,\partial_i\bigl(a_{ij}(u_p)_{kj}\bigr)
+a_{ij}(u_p)_{ki}(u_p)_{kj}\Bigr]
=\sum_ka_{ij}(u_p)_{ki}(u_p)_{kj}\ \ge\ 0
\end{equation}
in $U$. Thus $v\in C^\infty(\overline U)\subset W^{1,2}(U)$ is a subsolution
in $U$ of a linear divergence-form operator with bounded, uniformly elliptic
coefficients and no lower-order terms, and it attains at $x_0\in U$ its
maximum over $U$. By the strong maximum principle \cite[Theorem~8.19]{GT},
$v$ is constant on $U$. Then the left-hand side of \eqref{eq:subsolution}
vanishes, so $a_{ij}(u_p)_{ki}(u_p)_{kj}=0$ for every $k$, and positive
definiteness forces $D^2u_p=0$ on $U$. But then
$\Delta_pu_p=a_{ij}(u_p)_{ij}=0$ on $U$, contradicting $-\Delta_pu_p=1$.
\end{proof}

The exterior barrier agrees with the boundary datum at the tangency point, so
its inward normal derivative bounds that of $u_p$.
\begin{lemma}[Boundary gradient estimate]\label{lem:boundary-gradient}
Let $\Omega\subset\R^n$ be a bounded $C^2$ domain with uniform exterior ball
radius $r_e>0$. Then, for every $p>1$ and every $y\in\partial\Omega$,
$\abs{\nabla u_p(y)}^{\,p-1}\le K_\Omega$.
\end{lemma}
\begin{proof}
Let $R>D+r_e$ and keep the notation of Lemma~\ref{lem:barrier}. Let
$\tau:=-\nu(y)$ be the inward unit normal. For small $t>0$ the point $y+t\tau$
lies in $\Omega$ with $\abs{y+t\tau-c_y}=r_e+t$, and $u_p(y)=W_{y,R}(y)=0$, so
\[
\frac{u_p(y+t\tau)}{t}\ \le\ \frac{W_{y,R}(y+t\tau)}{t}
=\frac1t\int_{r_e}^{r_e+t}F_R(s)^{1/(p-1)}\,ds.
\]
Letting $t\downarrow0$ and using $u_p\in C^1(\overline\Omega)$ together with
the continuity of $F_R$ gives $\partial_\tau u_p(y)\le F_R(r_e)^{1/(p-1)}$.
Because $u_p$ vanishes on $\partial\Omega$ its tangential derivatives vanish
there, and $u_p\ge0$ gives $\partial_\tau u_p(y)\ge0$; hence
$\abs{\nabla u_p(y)}=\partial_\tau u_p(y)\le F_R(r_e)^{1/(p-1)}$. Let
$R\downarrow D+r_e$.
\end{proof}
\begin{proof}[Proof of Theorem~\ref{thm:gradient-bound}]
Combine Lemma~\ref{lem:grad-max} with Lemma~\ref{lem:boundary-gradient}.
\end{proof}

The limit of the gradient supremum follows; that it equals $1$ does not imply
uniform convergence of $\nabla u_p$ to $\nabla d$
(Proposition~\ref{prop:ball-gradient}).
\begin{corollary}[Limit of the gradient supremum]\label{cor:gradient-limit}
Let $\Omega$ be a bounded $C^2$ domain. Then
$\norm{\nabla u_p}_{L^\infty(\Omega)}\to1$ as $p\to\infty$.
\end{corollary}
\begin{proof}
Theorem~\ref{thm:gradient-bound} gives
$\limsup_pL_p\le\lim_p(K_\Omega^+)^{1/(p-1)}=1$. Conversely, fix $x_0$ with
$d(x_0)=\rho_\Omega$; Lemma~\ref{lem:segment} gives
$L_p\ge u_p(x_0)/\rho_\Omega$, and $u_p(x_0)\to\rho_\Omega$ by
Theorem~\ref{thm:universal-rate}.
\end{proof}

\section{Consequences for related quantities}\label{sec:consequences}
Some of the literature works with the normalized torsion function
$v_p:=u_p/k_p$, $k_p:=\norm{\nabla u_p}_{L^p(\Omega)}$, for which
$\norm{\nabla v_p}_{L^p(\Omega)}=1$ and $-\Delta_pv_p=k_p^{1-p}$. The
mean-to-max ratio of the $p$-torsion function, studied by Briani and
Bucur~\cite{BB} in connection with honeycomb structures, is
\[
\Phi_p(\Omega):=\frac{1}{\abs{\Omega}\,\norm{u_p}_{L^\infty(\Omega)}}
\int_\Omega u_p\,dx,
\qquad
\Phi_\infty(\Omega):=\frac{1}{\abs{\Omega}\,\rho_\Omega}\int_\Omega d\,dx.
\]
\begin{corollary}[Normalized torsion and mean-to-max ratio]\label{cor:consequences}
Suppose $\norm{u_p-d}_{L^\infty(\Omega)}\le C_\ast/p$ for all $p\ge p_\ast$.
Then there are $p_2=p_2(\Omega,C_\ast,p_\ast)$ and $C=C(\Omega,C_\ast)>0$ such
that, for all $p\ge p_2$,
\[
\norm{v_p-d}_{L^\infty(\Omega)}\le\frac{C}{p},
\qquad
\abs{\Phi_p(\Omega)-\Phi_\infty(\Omega)}\le\frac{C}{p}.
\]
With $C_\ast=C_{n,D}$ and $p_\ast=p_0$ from Theorem~\ref{thm:universal-rate},
$p_2$ and $C$ depend only on $n$ and $\Omega$.
\end{corollary}
\begin{proof}
Testing \eqref{eq:weak} with $u_p$ gives $k_p^p=\int_\Omega u_p>0$. Put
$I:=\int_\Omega d>0$ and
\[
p_2:=\max\Bigl\{p_\ast,\ \frac{2C_\ast\abs\Omega}{I},\ \frac{2C_\ast}{\rho_\Omega}\Bigr\}.
\]
For $p\ge p_2$ the hypothesis gives $\tfrac12I\le\int_\Omega u_p\le2I$ and
$\norm{u_p}_{L^\infty}\ge\rho_\Omega/2$, hence
$\abs{\log k_p}=\frac1p\abs{\log\int_\Omega u_p}\le C/p$ and
$\abs{k_p^{-1}-1}\le C/p$. Since $\norm{u_p}_{L^\infty}\le\rho_\Omega+C_\ast/p$,
\[
\norm{v_p-d}_\infty\le\abs{k_p^{-1}-1}\norm{u_p}_\infty
+\norm{u_p-d}_\infty\le\frac{C}{p}.
\]
For the ratio, write $\Phi_p=m_p/M_p$ and $\Phi_\infty=m_\infty/M_\infty$
with $m_p:=\abs{\Omega}^{-1}\int_\Omega u_p$, $M_p:=\norm{u_p}_\infty$,
$m_\infty:=\abs{\Omega}^{-1}\int_\Omega d$ and $M_\infty:=\rho_\Omega$. Then
$\abs{m_p-m_\infty}\le C_\ast/p$, $\abs{M_p-M_\infty}\le C_\ast/p$,
$M_p\ge\rho_\Omega/2$ for $p\ge p_2$, and $m_\infty\le M_\infty$, so
\[
\abs{\Phi_p-\Phi_\infty}
\le\frac{\abs{m_p-m_\infty}}{M_p}
+m_\infty\frac{\abs{M_p-M_\infty}}{M_pM_\infty}
\le\frac{4C_\ast}{\rho_\Omega\,p}. \qedhere
\]
\end{proof}

\section{Quantitative convergence of the gradients}\label{sec:gradients}
The estimates of this section quantify the strong $L^q$ convergence of
$\nabla u_p$ proved in~\cite{BLR} by combining the uniform error of
Theorem~\ref{thm:universal-rate} with the finite-measure structure of
$-\Delta d$ from Lemma~\ref{lem:tv-exterior-ball}. The following monotonicity
inequality converts the energy defect into an $L^2$ gradient defect, using
$\abs{\nabla d}=1$ almost everywhere.
\begin{lemma}\label{lem:monotonicity}
Let $p\ge2$, $\xi\in\R^n$ and $\abs{\eta}=1$. Then
\begin{equation}\label{eq:monotonicity}
\bigl(\abs{\xi}^{p-2}\xi-\eta\bigr)\cdot(\xi-\eta)
\ \ge\ \tfrac12\abs{\xi-\eta}^{2}.
\end{equation}
\end{lemma}
\begin{proof}
For $\xi=0$ the left-hand side equals $1$ and the right-hand side $1/2$, so
assume $r:=\abs{\xi}>0$ and put $c:=(\xi\cdot\eta)/r\in[-1,1]$. Twice the
left-hand side minus $\abs{\xi-\eta}^2$ equals $2r^p+1-r^2-2r^{p-1}c$, which
is minimized over $c\in[-1,1]$ at $c=1$, where it equals
$2r^p-2r^{p-1}+1-r^2=(r-1)\bigl(2r^{p-1}-r-1\bigr)$. For $r\ge1$ we have
$r^{p-1}\ge r$, so both factors are non-negative; for $0<r\le1$ we have
$r^{p-1}\le r$, so both are non-positive. In either case the product is
non-negative.
\end{proof}

\begin{theorem}[Quantitative gradient convergence]\label{thm:gradient-rate}
Let $\Omega\subset\R^n$ be a bounded domain such that $\mu:=-\Delta d$
is a signed Radon measure on $\Omega$ with $\abs{\mu}(\Omega)<\infty$, and let
$p_0$ and $C_{n,D}$ be as in Theorem~\ref{thm:universal-rate}. Then, for every
$p\ge p_0$,
\begin{equation}\label{eq:L2-rate}
\norm{\nabla u_p-\nabla d}_{L^2(\Omega)}\ \le\ \frac{C}{\sqrt p},
\qquad C:=\Bigl(2C_{n,D}\bigl(\abs\Omega+\abs\mu(\Omega)\bigr)\Bigr)^{1/2},
\end{equation}
and consequently
$\norm{\nabla u_p-\nabla d}_{L^q(\Omega)}\le C\abs\Omega^{1/q-1/2}p^{-1/2}$ for
each $q\in[1,2]$. For each $q\in[2,\infty)$ there is
$C_q=C_q(n,D,\abs\Omega,\abs\mu(\Omega),q)$ such that, for every
$p\ge\max\{q,p_0\}$,
\begin{equation}\label{eq:Lq-rate}
\norm{\nabla u_p-\nabla d}_{L^q(\Omega)}\ \le\ C_q\,p^{-1/q}.
\end{equation}
In particular, by Lemma~\ref{lem:tv-exterior-ball}, these conclusions hold on
every bounded Lipschitz domain with a uniform exterior ball, hence on every
bounded convex domain and on every bounded $C^2$ domain.
\end{theorem}
\begin{proof}
Let $p\ge p_0$; then $p\ge2n>n$ and $p\ge2$. Put $e_p:=u_p-d$. By
Lemma~\ref{lem:d-in-W1p}, $e_p\in W^{1,p}_0(\Omega)$; since $p>n$, its
extension by zero to $\R^n$ is continuous by the Morrey embedding. Testing
\eqref{eq:weak} with $e_p$,
\begin{equation}\label{eq:test1}
\int_\Omega\abs{\nabla u_p}^{p-2}\nabla u_p\cdot\nabla e_p\,dx
=\int_\Omega e_p\,dx.
\end{equation}
Since $\abs{\nabla d}=1$ almost everywhere,
$\abs{\nabla d}^{p-2}\nabla d=\nabla d$, and the finite-measure hypothesis
gives
\begin{equation}\label{eq:test2}
\int_\Omega\nabla d\cdot\nabla e_p\,dx=\int_\Omega e_p\,d\mu .
\end{equation}
Identity \eqref{eq:test2} holds by definition for $e_p\in C_c^\infty(\Omega)$;
for general $e_p$, approximate in $W^{1,p}_0(\Omega)$ by functions in
$C_c^\infty(\Omega)$, extend the approximants by zero to $\R^n$, and apply the
Morrey embedding there; since $p>n$ this gives uniform convergence as well, so
both the Lebesgue integral and the pairing against the finite measure $\mu$
pass to the limit. Subtracting \eqref{eq:test2} from \eqref{eq:test1},
\[
I_p:=\int_\Omega
\bigl(\abs{\nabla u_p}^{p-2}\nabla u_p-\nabla d\bigr)\cdot
\bigl(\nabla u_p-\nabla d\bigr)\,dx
=\int_\Omega e_p\,dx-\int_\Omega e_p\,d\mu ,
\]
and $I_p\ge0$ by Lemma~\ref{lem:monotonicity}. Theorem~\ref{thm:universal-rate}
gives
\[
0\ \le\ I_p\ \le\ \bigl(\abs{\Omega}+\abs{\mu}(\Omega)\bigr)
\norm{e_p}_{L^\infty(\Omega)}\ \le\ \frac{C_{n,D}\bigl(\abs{\Omega}+\abs{\mu}(\Omega)\bigr)}{p}.
\]
Applying \eqref{eq:monotonicity} pointwise with $\xi=\nabla u_p$ and
$\eta=\nabla d$ gives $\tfrac12\norm{\nabla u_p-\nabla d}_{L^2}^2\le I_p$,
which is \eqref{eq:L2-rate}; for $1\le q\le2$, H\"older's inequality gives
the stated bound.

Fix $q\ge2$ and take $p\ge\max\{q,p_0\}$. Write $h_p:=\abs{\nabla u_p-\nabla d}$ and
$E_p:=\{\abs{\nabla u_p}>2\}$. On $\Omega\setminus E_p$ one has $h_p\le3$;
on $E_p$ one has $h_p\le\tfrac32\abs{\nabla u_p}$ and
$\abs{\nabla u_p}^q\le2^{q-p}\abs{\nabla u_p}^p$. Hence
\[
\int_\Omega h_p^q\,dx
\le 3^{q-2}\int_\Omega h_p^2\,dx
 +(3/2)^q\int_{E_p}\abs{\nabla u_p}^q\,dx
\le 3^{q-2}\norm{\nabla u_p-\nabla d}_{L^2}^2
 +3^q2^{-p}\int_\Omega\abs{\nabla u_p}^p\,dx.
\]
The energy identity $\int_\Omega\abs{\nabla u_p}^p=\int_\Omega u_p$ and
Theorem~\ref{thm:universal-rate} give
$\int_\Omega\abs{\nabla u_p}^p\le\abs\Omega\,(D/2+C_{n,D})$ for $p\ge p_0$,
and $2^{-p}\le1/p$, so the right-hand side is at most $C_q/p$ with
$C_q$ depending only on $n$, $D$, $\abs\Omega$, $\abs\mu(\Omega)$ and $q$.
Taking $q$-th roots proves \eqref{eq:Lq-rate}.
\end{proof}

\begin{remark}\label{rem:gradient-not-sharp}
We do not claim that the exponents in \eqref{eq:L2-rate} and
\eqref{eq:Lq-rate} are sharp: on a ball, Proposition~\ref{prop:ball-gradient}
gives $\norm{\nabla u_p-\nabla d}_{L^q}=\Theta(1/p)$. The loss is in
Lemma~\ref{lem:monotonicity}, which retains only the transversal eigenvalue
$\abs\xi^{p-2}$ of the differential of $a(\xi)=\abs\xi^{p-2}\xi$ and not the
radial one $(p-1)\abs\xi^{p-2}$; for $\xi=(1-\varepsilon)\eta$ with
$(p-1)\varepsilon$ small, the left-hand side of \eqref{eq:monotonicity} is
$\approx(p-1)\varepsilon^{2}$ while the right-hand side is $\varepsilon^{2}/2$.
On a ball $\nabla u_p-\nabla d$ is parallel to $\nabla d$, so the ball exhibits
exactly this loss; whether the exponent $1/2$ is attained on some domain is
open.
\end{remark}
\section{The ball: exact error, second order, and gradients}\label{sec:ball}
The Dirichlet $p$-torsion profile on a ball is given explicitly in van den
Berg--Bucur~\cite[Section~4.2]{VdBB2014}, and the coefficients $a_n$, $b_n$
below agree with Fayolle--Belyaev~\cite[Appendix~B]{FB2026}; what is new is
the exact uniform error, the uniformity of the remainder up to the medial axis
$\{r=0\}$, and the gradient asymptotics.
\begin{theorem}[Exact error and uniform asymptotics on the ball]\label{thm:ball}
Let $\Omega=B(0,1)\subset\R^n$, $n\ge1$, write $r=\abs{x}$, $d(r)=1-r$ and
$\ell:=\log n$. Then, with $A_{p,n}$ as in \eqref{eq:Apn},
\begin{equation}\label{eq:ball-profile}
u_p(r)=A_{p,n}\bigl(1-r^{\gamma}\bigr),
\end{equation}
one has $u_p\le d$ on $\overline{B(0,1)}$, and
\begin{equation}\label{eq:ball-exact}
\norm{u_p-d}_{L^\infty(B(0,1))}=1-A_{p,n}
\quad\text{for every }p>1,
\qquad
1-A_{p,n}
=\frac{1+\ell}{p}-\frac{\ell^{2}}{2p^{2}}+O\bigl(p^{-3}\bigr).
\end{equation}
Moreover, the pointwise expansion
\begin{equation}\label{eq:ball-second-order}
u_p(r)=d(r)+\frac{a_n(r)}{p}+\frac{b_n(r)}{p^{2}}+O\bigl(p^{-3}\bigr)
\end{equation}
holds uniformly for $r\in[0,1]$, where
\[
a_n(r)=-(1+\ell)(1-r)-r\log r,
\qquad
b_n(r)=\frac{\ell^{2}}{2}(1-r)+\ell\,r\log r-\frac12\,r(\log r)^{2},
\]
with the continuous conventions $r\log r=r(\log r)^2=0$ at $r=0$.
\end{theorem}
\begin{proof}
In the proof of Lemma~\ref{lem:ball-comparison}, with $x=0$ and $r=1$, it
was verified that $w(y)=A_{p,n}(1-\abs{y}^{\gamma})$ solves $-\Delta_pw=1$
weakly in $B(0,1)$ and belongs to $W^{1,p}_0(B(0,1))$; by uniqueness of the
weak solution, $u_p=w$, which is \eqref{eq:ball-profile}. Put $h_p:=d-u_p$.
Since $A_{p,n}\gamma=n^{-1/(p-1)}$ and $\gamma-1=1/(p-1)$,
\[
h_p'(r)=-1+A_{p,n}\gamma\,r^{\gamma-1}=-1+\Bigl(\frac rn\Bigr)^{1/(p-1)}\le0
\qquad(0\le r\le1,\ n\ge1).
\]
Thus $h_p$ is non-increasing with $h_p(1)=0$, so $0\le h_p\le h_p(0)=1-A_{p,n}$,
which gives $u_p\le d$ and the exact identity in \eqref{eq:ball-exact}. For
the expansions,
\[
\log A_{p,n}=\log\Bigl(1-\frac1p\Bigr)-\frac{\ell}{p-1}
=-\frac{1+\ell}{p}-\frac{\tfrac12+\ell}{p^{2}}+O\bigl(p^{-3}\bigr),
\]
so that, since $\tfrac12(1+\ell)^2-(\tfrac12+\ell)=\tfrac12\ell^{2}$,
\begin{equation}\label{eq:A-expansion}
A_{p,n}=1-\frac{1+\ell}{p}+\frac{\ell^{2}}{2p^{2}}+O\bigl(p^{-3}\bigr),
\end{equation}
which is the expansion in \eqref{eq:ball-exact}. For the radial power, write
$s:=-\log r\ge0$, so $r^{\gamma}=e^{-s}e^{-s/(p-1)}$. Taylor's formula with
remainder, multiplied by $e^{-s}$, is uniform in $s\ge0$ because
$\sup_{s\ge0}e^{-s}s^{k}<\infty$ for each fixed $k$; with
$(p-1)^{-1}=p^{-1}+p^{-2}+O(p^{-3})$ this gives
\begin{equation}\label{eq:r-expansion}
r^{\gamma}=r+\frac{r\log r}{p}
+\frac{r\bigl(\log r+\tfrac12(\log r)^{2}\bigr)}{p^{2}}+O\bigl(p^{-3}\bigr)
\end{equation}
uniformly on $[0,1]$. Inserting \eqref{eq:A-expansion} and
\eqref{eq:r-expansion} into $u_p(r)=A_{p,n}(1-r^{\gamma})$ and subtracting
$d(r)=1-r$ gives \eqref{eq:ball-second-order}.
\end{proof}

\begin{corollary}[Optimal leading constant on the ball]\label{cor:ball-constant}
On the unit ball,
\[
\lim_{p\to\infty}p\,\norm{u_p-d}_{L^\infty(B(0,1))}=1+\log n .
\]
In particular the order $1/p$ in Theorems~\ref{thm:universal-rate}
and~\ref{thm:main} cannot be improved in general.
\end{corollary}

The explicit profile gives the exact first-order gradient asymptotics in every
finite $L^q$ space and the $L^\infty$ obstruction.
\begin{proposition}[Gradient asymptotics on the ball]\label{prop:ball-gradient}
Let $\Omega=B(0,1)\subset\R^n$, $n\ge1$, and $1\le q<\infty$. Then
\begin{equation}\label{eq:ball-gradient-q}
\lim_{p\to\infty}p\,\norm{\nabla u_p-\nabla d}_{L^q(B(0,1))}
=\Bigl(\abs{\partial B(0,1)}\int_0^1
\Bigl(\log\frac nr\Bigr)^{q}r^{\,n-1}\,dr\Bigr)^{1/q},
\end{equation}
where $\abs{\partial B(0,1)}=\mathcal H^{n-1}(\partial B(0,1))$, with
$\abs{\partial B(0,1)}=2$ for $n=1$, whereas
\begin{equation}\label{eq:ball-gradient-infty}
\norm{\nabla u_p-\nabla d}_{L^\infty(B(0,1))}=1
\qquad\text{for every }p>1.
\end{equation}
\end{proposition}
\begin{proof}
For $0<r<1$, \eqref{eq:ball-profile} gives $\nabla u_p=u_p'(r)\hat x$ with
$u_p'(r)=-(r/n)^{1/(p-1)}$, while $\nabla d=-\hat x$; hence
\[
\abs{\nabla u_p-\nabla d}=1-\Bigl(\frac rn\Bigr)^{1/(p-1)}
=1-e^{-\log(n/r)/(p-1)}\ \ge\ 0.
\]
For fixed $r>0$, $p$ times this tends to $\log(n/r)$; and since
$1-e^{-t}\le t$, for $p\ge2$ it is dominated by $2\log(n/r)$, which lies in
$L^q\bigl((0,1),r^{n-1}dr\bigr)$. Dominated convergence gives
\eqref{eq:ball-gradient-q}. For every $\varepsilon\in(0,1)$ the integrand
exceeds $1-\varepsilon$ on a set of positive measure near the origin, so the
essential supremum equals $1$.
\end{proof}

\begin{remark}\label{rem:ball-lessons}
For $n=1$ the ball is the interval $(-1,1)$ and
$u_p(x)=\frac{p-1}{p}\bigl(1-\abs x^{\gamma}\bigr)$, so
$\abs{u_p'(x)}^{p-1}=\abs x$ and $\norm{u_p'}_{L^\infty(-1,1)}^{\,p-1}=1$,
whereas $K_\Omega=D=2$: the gradient bound \eqref{eq:gradient-bound} holds in
dimension one with room to spare. In this case
\eqref{eq:ball-gradient-infty} is the one-dimensional example of~\cite{BLR}
for the failure of uniform gradient convergence. In every dimension it shows
that $\norm{\nabla u_p}_{L^\infty}\to1$ (Corollary~\ref{cor:gradient-limit})
does not imply $\nabla u_p\to\nabla d$ in $L^\infty$: the defect concentrates
at the cut locus $\{r=0\}$. An $O(1/p)$ gradient rate on compact subsets of
$\Omega\setminus\Sigma$ is plausible but does not follow from interpolation.
\end{remark}

\section*{Declarations}
\noindent\textbf{Conflict of interest.} The authors declare that they have no
conflict of interest.

\noindent\textbf{Data availability.} Data sharing is not applicable to this
article, as no data sets were generated or analysed during the current study.

\end{document}